\documentclass [11pt]{article}
\usepackage{amssymb,amsmath,comment,amsthm}

\def\N{\mathbb{N}}

\def\F{\mathbb{F}}

\newtheorem{theorem}{Theorem}[section]
\newtheorem{proposition}[theorem]{Proposition}
\newtheorem{corollary}[theorem]{Corollary}
\newtheorem{lemma}[theorem]{Lemma}

\newtheorem{remark}[theorem]{Remark}

\begin{document}
\title{Non-splitting bi-unitary perfect polynomials  over~$\F_4$ with less than five prime factors}
 \author{Olivier Rahavandrainy\\
Univ Brest, UMR CNRS 6205\\
Laboratoire de Math\'ematiques de Bretagne Atlantique}
\maketitle
\\
{\bf{Abstract}}
We identify all non-splitting bi-unitary perfect polynomials 
over the field~$\F_4$, which admit  at most four irreducible divisors. There is an infinite number of such divisors.
{\section{Introduction}}
In this paper, we work over the finite field  $\F_4$ of $4$ elements:
$$\text{$\F_4 = \{0,1,\alpha,\alpha+1\}$ where $\alpha^2+\alpha+1 = 0$.}$$
As usual, $\N$ (resp. $\N\sp{*}$) denotes the set of nonnegative
integers (resp. of positive integers).

Throughout the paper, every polynomial is a monic one.

Let $S \in \F_4[x]$ be a nonzero polynomial. 
A divisor $D$ of $S$ is called unitary if $\gcd(D,S/D)=1$. We designate by $\gcd_u(S,T)$ the greatest
common unitary divisor of $S$ and $T$.
A divisor $D$ of $S$ is called bi-unitary if $\gcd_u(D,S/D)=1$.
We denote by $\sigma(S)$ (resp. $\sigma^*(S)$, $\sigma^{**}(S)$) the sum of all divisors
(resp. unitary divisors, bi-unitary divisors)
of $S$.
The functions $\sigma$, $\sigma^*$ and $\sigma^{**}$ are all multiplicative.
We say that $S$ is \emph{perfect} (resp. \emph{unitary perfect},
\emph{bi-unitary perfect}) if $\sigma(S) = S$
(resp. $\sigma^*(S)=S$, $\sigma^{**}(S)=S$).

Finally, we say that $A$ is {\it{indecomposable}} bi-unitary perfect if $A$ has no proper divisor which is  bi-unitary perfect.

Several studies are done about perfect, unitary and bi-unitary perfect polynomials (see \cite{BeardU}, \cite{BeardBiunitary}, \cite{Canaday}, \cite{Gall-Rahav-F4}, 
\cite{Gall-Rahav9}, \cite{Gall-Rahav2},  \cite{Gall-Rahav10} and references therein).

In this paper, we are interested in non-splitting polynomials over $\F_4$ which are bi-unitary perfect (b.u.p.) and divisble by $r$ irreducible factors, 
where $r \leq 4$.

The splitting case is already treated in (\cite{Gall-Rahav-split-Fp2}, Proposition 3.1 and Theorem 3.2). However, we better precise  these results in Theorem \ref{casebupsplit}.

We consider the two following sets:
$$\begin{array}{l}
\Omega_1:=\{P \in \F_4[x]: P \text{ and $P+1$ are both irreducible}\}\\ 
\Omega_2:=\{P \in \F_4[x]: ¨P, \text{$P+1$, $P^3+P+1$ and $P^3+P^2+1$ are all irreducible}\}.
\end{array}$$
We see that $\Omega_2 \subset \Omega_1$, $\Omega_2$ contains the four (monic) monomials of $\F_4[x]$ and it is an infinite set (\cite{Gall-Rahav8}, Lemma 2). For example, for any $k \in \N$, 
$P_k:= x^{2\cdot 5^k} +  x^{5^k} +\alpha \in \Omega_2$.

We get the following two results related to the fact that $A$ splits or not.

\begin{theorem} \label{casebupsplit}
Let $A = x^a(x+1)^b(x+\alpha)^c(x+\alpha+1)^d \in \F_4[x]$, where $a,b,c,d \in \N$ are not all odd. Then, $A$ is b.u.p if and only if one of the following conditions holds:\\
i) $a=b=c=d=2$,\\
ii) $a=b=2$ and $c=d=2^n-1$, for some $n \in \N$,\\
iii) $a=b=2^n-1$ and $c=d=2$, for some $n \in \N$,\\
iv) $a, b,c,d$ are given by Table $(\ref{expovalues})$.
\end{theorem}
\begin{equation} \label{expovalues}
\begin{array}{|l|c|c|c|c|c|c|c|c|c|c|c|c|c|c|}
\hline
a&4&4&4&4&4&4&5&5&5&5&6&6&6&6\\
\hline
b&3&3&4&4&4&4&3&3&4&4&6&6&6&6\\
\hline
c&3&4&3&4&5&6&4&6&4&5&3&4&5&6\\
\hline
d&4&3&5&4&3&6&4&6&5&4&5&4&3&6\\
\hline
\end{array}
\end{equation}

\begin{theorem} \label{casebupnon-split}
Let $A =P^aQ^bR^cS^d \in \F_4[x]$, where $A$ does not split and $a,b,c,d$ are not all odd. Then, $A$ is b.u.p if and only if  one of the following conditions holds:\\
i) $a=b=c=d=2$, $P, R \in \Omega_1$, $Q = P+1$ and $S = R+1$,\\
ii) $a=b=2$, $c=d=2^n-1$, for some $n \in \N$,  $P, R \in \Omega_1$, $Q = P+1$ and $S = R+1$,\\
iii) $a=b=2^n-1$, for some $n \in \N$,  $c=d=2$,  $P, R \in \Omega_1$, $Q = P+1$ and $S = R+1$,\\
iv) $P \in \Omega_2$, $Q = P+1,\ R, S \in \{P^3+P+1 , P^3+P^2+1\}$ and $(a,b,c,d) \in \{(7,13,2,2),(13,7,2,2),(14,14,2,2)\}$.
\end{theorem}
Note that if $a,b,c$ and $d$ are all odd, then  $\sigma^{**}(A) = \sigma(A)$. So, $A$ is b.u.p. if and only if $A$ is perfect. We also see that there exists no b.u.p. polynomial $A$ with $\omega(A) = 3$.

Since $\Omega_1$ and $\Omega_2$ are infinite sets, we see that there are infinitely many indecomposable and odd b.u.p. polynomials over $\F_4$, even if there are only three 4-tuples available exponents.

\section{Preliminaries}\label{preliminairebup}
Some of  the following results are obvious or (well) known, so we omit
their proofs. See also \cite{Gall-Rahav-bup-Mersenne}. 
\begin{lemma} \label{gcdunitary}
Let $T$ be an irreducible polynomial over $\F_4$ and $k,l\in \N^*$.
Then,  $\gcd_u(T^k,T^l) = 1 \ ($resp. $T^k)$ if $k \not= l \ ($resp. $k=l)$.\\
In particular, $\gcd_u(T^k,T^{2n-k}) = 1$ for $k \not=n$,  $\gcd_u(T^k,T^{2n+1-k}) = 1$ for
any $0\leq k \leq 2n+1$.
\end{lemma}
\begin{lemma} \label{aboutsigmastar2}
Let $T \in \F_4[x]$ be irreducible. Then\\
i) $\sigma^{**}(T^{2n}) = (1+T)\sigma(T^n) \sigma(T^{n-1}), \
\sigma^{**}(T^{2n+1}) = \sigma(T^{2n+1})$.\\
ii) For any $c \in \N$, $1+T$ divides $\sigma^{**}(T^{c})$ but $T$ does not.
\end{lemma}
\begin{corollary} \label{expo-all-odd}
Let $A = P^h Q^kR^l S^t$ be such that $h,k,l$ and $t$ are all odd. Then, $A$ is b.u.p. if and only if it is perfect.
\end{corollary}

\begin{lemma} \label{multiplicativity}
If $A = A_1A_2$ is b.u.p. over $\F_4$ and if
$\gcd(A_1,A_2) =~1$, then $A_1$ is b.u.p. if and only if
$A_2$ is b.u.p.
\end{lemma}
\begin{lemma} \label{translation}
If $A$ is b.u.p. over $\F_4$, then the polynomial $A(x+\lambda)$
is also b.u.p. over~$\F_4$, for any $\lambda \in \{1,\alpha,\alpha+1\}$.
\end{lemma}
\begin{lemma} \label{splitcritere}
i) $\sigma^{**}(x^{2k})$ splits over $\F_4$ if and only if \ $2k \in \{2,4,6\}$.\\
ii) $\sigma^{**}(x^{2k+1})$ splits over $\F_4$ if and only if \ $2k+1 =N\cdot 2^n-1$ where $N \in \{1,3\}$.
\end{lemma}
\begin{remark}
{\emph{We get from Lemma \ref{splitcritere} and for an irreducible polynomial~$T$:
\begin{equation} \label{relation1}
\left\{\begin{array}{l}
\sigma^{**}(T^2) = (T+1)^2 \ \ \ \ \ \ \ \ \ \ \ \ \ \ \ \ \ \ \ \ \  \ \ \ \ \ \ \ \ \ \ \ \ \ \ \ \ \ \ \ \ \ (i)\\
\sigma^{**}(T^4) = (T+1)^2 (T+\alpha)(T+\alpha+1) \ \ \ \ \ \ \ \ \ \ \ \ \ \ \ \ \  (ii)\\
\sigma^{**}(T^6) = (T+1)^4 (T+\alpha)(T+\alpha+1) \ \ \  \ \ \ \ \ \ \ \ \ \ \ \ \ \ (iii)\\
\sigma^{**}(T^{2^n-1}) = (T+1)^{2^n-1}\ \ \ \ \ \ \ \ \ \ \ \ \ \ \ \ \ \ \ \ \ \ \ \ \ \ \ \ \ \ \ \ \ \ (iv)\\
\sigma^{**}(T^{3 \cdot 2^n-1}) = (T+1)^{2^n-1}(T+\alpha)^{2^n}(T+\alpha+1)^{2^n} \ \ (v)
\end{array}
\right.
\end{equation}}}
\end{remark}
We sometimes use the above equalities for a suitable $T$. 
\section{Proof of Theorem \ref{casebupsplit}}
This theorem is already stated in \cite{Gall-Rahav-split-Fp2}. We do not rewrite its proof.

Lemma \ref{perfectresult} below completes Theorem 3.4 in \cite{Gall-Rahav-F4}, where  one family of splitting perfect polynomials over $\F_4$ was missing.
 
See \cite{Gall-Rahav3} and \cite{Gall-Rahav8} for the non-splitting case.
\begin{lemma} \label{perfectresult}
The polynomial $x^h(x+1)^k(x+\alpha)^l(x+\alpha+1)^t$ is perfect over $\F_4$ if and only if one of the following conditions is satisfied:\\
i) $h=k=2^n-1,\ l=t=2^m-1$, for some $n,m \in \N$,\\
ii) $h=k=l=t=N \cdot 2^n-1$, for some $n\in \N$ and $N \in \{1,3\}$,\\
iii) $h=l=3 \cdot 2^r-1, \ k=t=2\cdot 2^r-1$, for some $r \in \N$,\\
iv) $h=k=3\cdot 2^r-1,\ l=6\cdot 2^r-1, t = 4 \cdot 2^r-1$ for some $r \in \N$.
\end{lemma}
\begin{proof}
Put $A = x^h(x+1)^k(x+\alpha)^l (x+\alpha+1)^t$. The sufficiency is obtained by direct computations.
 
For the necessity, we recall the following facts in  (\cite{Gall-Rahav-F4}, Lemma 2.7):
\begin{lemma}
Let $a \in \{h,k,l,t\}$. Then, $a = 3 \cdot 2^w-1$ if $a \equiv 2 \mod 3$ and  $a = 2^w-1$ if $a \not\equiv 2 \mod 3$, with $w \in \N$.\\
In particular, $a = 2$ if  $(a$ is even and  $a \equiv 2 \mod 3)$.
\end{lemma}

We take account of Lemmas 2.2, 2.3, 2.6 and 2.7 in \cite{Gall-Rahav-F4}, for the congruency modulo $3$ of each exponent. We get 16 possible cases according to $a \equiv 2 \mod 3$ or not.

Moreover, by Lemma 2.2 in \cite{Gall-Rahav-F4}, the $3$ maps $x \mapsto x+1$, $x \mapsto x+\alpha$ and $x \mapsto x+\alpha+1$ preserve 
 ``perfection''. 

Therefore, $h$ and $k$ (resp. $h$ and $l$, $h$ and $t$) play symmetric roles. It remains  4 cases:
$$\begin{array}{l}
i) \ h,k \not\equiv 2 \mod 3\\
ii) \ h,l \equiv 2 \mod 3 \mbox{ and } k,t \not\equiv 2 \mod 3\\
iii) \ h,k,l,t \equiv 2 \mod 3\\
iv) \ h,k,l \equiv 2 \mod 3 \mbox{ and } t \not\equiv 2 \mod 3.
\end{array}$$

The first three of them are already treated in the proof of (\cite{Gall-Rahav-F4}, Theorem 3.4). We got the families i), ii) and iii) in Lemma \ref{perfectresult}. 

Now, for the case iv), we may write: 
$$\mbox{$h = 3\cdot 2^r-1,\ k = 3\cdot 2^s-1,\ l = 3\cdot 2^u-1$ et $t = 2^v-1$, where 
$r,s,u,v \in \N$.}$$
Compute $\sigma(A) = \sigma(x^h) \cdot \sigma((x+1)^k) \cdot \sigma((x+\alpha)^l) \cdot \sigma((x+\alpha+1)^t)$. 
$$\begin{array}{l}
\sigma(x^h) = (x+1)^{2^r-1} \cdot (x+\alpha)^{2^r} \cdot (x+\alpha+1)^{2^r}\\
\sigma((x+1)^k) = x^{2^s-1} \cdot (x+\alpha)^{2^s} \cdot (x+\alpha+1)^{2^s}\\
\sigma((x+\alpha)^l) = (x+\alpha+1)^{2^u-1} \cdot x^{2^u} \cdot (x+1)^{2^u}\\
\sigma((x+\alpha+1)^t) = \sigma((x+\alpha+1)^{2^v-1}) = (x+\alpha)^{2^v-1}.
\end{array}$$
Since $\sigma(A) = A$, by comparing  exponents in $A$ and those of in $\sigma(A)$, we get:
$$\begin{array}{l}
2^u + 2^s-1 = h = 3 \cdot 2^r-1\\
2^u + 2^r-1 = k = 3 \cdot 2^s-1\\
2^r + 2^s + 2^v-1 = l = 3 \cdot 2^u-1\\
2^u + 2^s+2^u-1 = t = 2^v-1
\end{array}$$
It follows that $s= r, \ u = r+1$ et $v = r+2$. Thus, $h=k=3\cdot 2^r-1$, $l = 3 \cdot 2^{r+1}-1$ and
$t = 2^{r+2}-1$. We obtain the family iv).
\end{proof}

\section{Proof of Theorem \ref{casebupnon-split}}
The sufficiency is obtained by direct computations. Propositions \ref{propo-omega=2}, \ref{propo-omega=3} and \ref{lesexposants23} give the necessity. 

As usual, $\omega(S)$ denotes the number of distinct irreducible factors of a polynomial $S$.
\subsection{Case $\omega(A)=2$} \label{caseomega=2}
Put $A = P^hQ^k$ with $\deg(P) \leq \deg(Q)$.
\begin{proposition} \label{propo-omega=2}
If $A$ is b.u.p., then $Q=P+1$ and either $(h=k=2)$ or $(h=k=2^r-1$, for some $r \in \N)$.
\end{proposition}
\begin{proof}
We get: 
$$\mbox{$\sigma^{**}(P^h) \sigma^{**}(Q^k) = \sigma^{**}(A)=A=P^hQ^k$ and $\omega(\sigma^{**}(P^h)) = 1 = \omega(\sigma^{**}(Q^k))$.}$$
If $h$ and $k$ are both odd, then $A$ is perfect so that $Q = P+1$ and $h=k=2^r-1$ for some $r \in \N$.\\
Now, we may suppose that $h$ is even. If $h=2$, then $\sigma^{**}(P^h) = (1+P)^2$. Since $P$ does not divide $\sigma^{**}(P^h)$, one has 
$Q^k = (1+P)^2$ and thus $Q= P+1$ and $k=2$. If $h \geq 4$, then $\omega(\sigma^{**}(P^h)) \geq 2$, which is impossible.
\end{proof}
\subsection{Case $\omega(A)=3$} \label{caseomega=3}
Put $A = P^hQ^kR^l$ with $\deg(P) \leq \deg(Q) \leq \deg(R)$.  Suppose that 
$$\sigma^{**}(P^h) \sigma^{**}(Q^k)  \sigma^{**}(R^l)= \sigma^{**}(A)=A=P^hQ^kR^l.$$
\begin{lemma} \label{Q=P+1}
The polynomial $P+1$ is irreducible and $Q=P+1$.
\end{lemma}
\begin{proof}
The polynomial $1+P$ is divisible by  $Q$ or by $R$, since it divides $\sigma^{**}(P^h)$ (Lemma \ref{aboutsigmastar2}). We may suppose that 
$Q \mid (1+P)$. So, $\deg(Q) = \deg(P)$ and  $Q = P+1$.
\end{proof}

\begin{lemma} \label{omegabistar<3}
One has $\omega(\sigma^{**}(P^h)) \leq 2,\  \omega(\sigma^{**}(Q^k)) \leq 2,\  \omega(\sigma^{**}(R^l))\leq 2$. 
Moreover, if $h$ is even (resp. odd), then $h=2$ (resp. $h=2^r-1$, $r\in \N^*$).
\end{lemma}
\begin{proof}
Since $P$ does not divide $\sigma^{**}(P^h)$, at most $Q$ and $R$ divide it. Hence,  $\omega(\sigma^{**}(P^h)) \leq 2$. Similarly, we get 
$\omega(\sigma^{**}(Q^k)) \leq 2$ and $\omega(\sigma^{**}(R^l))\leq 2$.\\
- If $h=2n$ is even, then $2 \geq \omega(\sigma^{**}(P^{2n})) = \omega((1+P) \sigma(P^n) \sigma(P^{n-1}))$. So, $n=1$.\\
- If $h$ is odd. Put $h=2^ru-1$, with $u$ odd. One has:
$$2 \geq \omega(\sigma^{**}(P^{2^ru-1})) = \omega((1+P)^{2^r-1} \sigma(P^{u-1})).$$ So, $u=1$ 
because  $\omega(\sigma(P^{u-1}) ) \geq 2$ if $u \geq 3$.
\end{proof}
\begin{proposition} \label{propo-omega=3}
If $h,k$ and $l$ are not all odd, then $A$ is not b.u.p.
\end{proposition}

\begin{proof} By Lemma \ref{Q=P+1}, one has $Q= P+1$ and so $A = P^h(P+1)^kR^l$.\\
- If $h,k$ are all even, then $h=k=2$. Therefore,
$$(1+P)^2 (1+Q)^2 \sigma^{**}(R^l) =  \sigma^{**}(A) =A=P^2Q^2  R^l.$$
Hence, $\sigma^{**}(R^l) =R^l$. It is impossible.\\
- If $h$ is even, $k$ odd and $l$ even, then $h=l=2$, $k=2^r-1$. Therefore,
$$Q^2 P^{2^r-1} (1+R)^2 = (1+P)^2 (1+Q)^{2^r-1} (1+R)^2 =  \sigma^{**}(A) =A=P^2Q^{2^r-1}  R^2.$$
Hence, $R$ divides $PQ$. It is impossible.\\ 
- If $h$ is even, $k$ and $l$ odd, then $h=2$, $k=2^r-1$, $l=2^s-1$. One has:
$$Q^2 P^{2^r-1} (1+R)^{2^s-1} = (1+P)^2 (1+Q)^{2^r-1} (1+R)^{2^s-1} =  \sigma^{**}(A) =A=P^2Q^{2^r-1}  R^{2^s-1}.$$
Hence, $R$ divides $PQ$. It is impossible.
\end{proof}

\subsection{Case $\omega(A)=4$} \label{caseomega=4}
Put $A = P^hQ^kR^lS^t$ with $\deg(P) \leq \deg(Q) \leq \deg(R) \leq \deg(S)$. 

We suppose that $A$ is  b.u.p. and indecomposable (i.e.,  neither $P^hQ^k$ nor $R^lS^t$ are b.u.p). 
\begin{lemma} \label{lesPQRS}
One has: $Q = P+1, 1+R = P^{u_1}Q^{v_1}, 1+S =  P^{u_2}Q^{v_2}R^z$ where $u_1,v_1 \geq 1$ and $u_2,v_2, z \geq 0$.\\
Moreover, if $\deg(R) = \deg(S)$ then  $u_2,v_2 \geq 1$ and $z=0$.
\end{lemma}
\begin{proof}
The polynomial $1+P$ divides $\sigma^{**}(A) = A$, so $Q$ divides $1+P$ and thus, $Q = 1+P$ because $\deg(P) \leq \deg(Q)$.\\
Now, $1+R$ divides $\sigma^{**}(A) = A$, so $1+R = P^{u_1}Q^{v_1}S^{u_3}$ and $u_3 = 0$ because  $\deg(R) \leq \deg(S)$. 
Since $R = P^{u_1}Q^{v_1} +1$ is irreducible, we conclude that $u_1, v_1 \geq 1$ and $\gcd(u_1,v_1) = 1$.  By the same reason,   $1+S =  P^{u_2}Q^{v_2}R^z$ where $u_2,v_2, z \geq 0$ and $z$ may be positive.
\end{proof}
\subsubsection{Case $\deg(P) =1$} \label{omega=4-one-splits}
We may suppose that $P = x$. Lemma \ref{Q=P+1} implies that $Q = x+1$. Moreover, $\deg(S)= 1$ if $\deg(R) = 1$. So, $\deg(S) \geq \deg(R) > 1$.\\
We write: $A = x^h(x+1)^kR^lS^t$. The exponents $h$ and $k$ play symmetric roles.
\begin{lemma}  [\cite{Gall-Rahav3}, Lemma 2.6] \label{caspair=6}
If $1+x+\cdots+ x^{2w} = UV$, then $\deg(U) = \deg(V)$ and $U(0) = 1 = V(0)$. \\
Moreover, if $R$ and $S$ are both of the form $x^{u_1} (x+1)^{v_1} + 1$, then $2w=6$ and $U, V \in \{x^3+x+1, x^3+x^2+1\}$.
\end{lemma}
\begin{lemma} \label{caseheven}
If $h$ is even, then $h \in \{2,14\}.$ Moreover, $R,S \in \{x^3+x+1, x^3+x^2+1\}$ if $h = 14$.
\end{lemma}
\begin{proof}
$\bullet$ If  $h \in \{4,6\}$, then $x+\alpha$ and $x+\alpha+1$ both divide $\sigma^{**}(x^h)$ and thus, they divide $\sigma^{**}(A) = A$. 
So, $A$ splits, which is impossible.\\
$\bullet$ If $h = 2n \geq 8$, then  $\sigma^{**}(x^h) = (1+x) \sigma(x^n) \sigma(x^{n-1})$.\\
- If $n=2w \geq 4$, then $\sigma(x^n) = RS$ because it divides $\sigma^{**}(A)=A$ and neither $x$ nor $x+1$ divide $\sigma(x^n)$.
So, by Lemma \ref{caspair=6}, $\deg(R) = \deg(S)$ and $R(0) = 1 = S(0)$. From Lemma \ref{lesPQRS}, we may put 
$1+R = x^{u_1}(x+1)^{v_1}$, $1+S =  x^{u_2}(x+1)^{v_2}$, where $u_1,u_2,v_1,v_2 \geq 1$. 
Therefore, $2w = 6$ and $h=12$. But, the monomials $x+1,x+\alpha$ and $x+\alpha+1$ all divide $\sigma^{**}(x^{12})$. It  contradicts the fact that $A$ does not split.\\
- If $n=2w+1$ is odd, then $\sigma(x^{n-1}) = RS$ and as above, $n-1=2w = 6$. So, $h=14$, 
$\sigma^{**}(x^{14}) = (x+1)^8 RS$ where $R,S \in \{x^3+x+1,x^3+x^2+1\}$. 
\end{proof}
\begin{lemma}  \label{casehodd}
If $h$ is odd, then $h = 2^ru-1$ where $r \in \N^*$ and $u \in \{1,7\}.$
\end{lemma}
\begin{proof}
Put $h = 2^ru-1$ with $u$ odd. One has: $$\sigma^{**}(x^h) = \sigma(x^h) = (1+x)^{2^r-1} [\sigma(x^{u-1})]^{2^r}.$$
If $u \geq 3$, then  $\sigma(x^{u-1}) = RS$. So, as we have just seen above, $u-1=6$ and $R, S \in \{x^3+x+1, x^3+x^2+1\}$.
\end{proof}
\begin{lemma} \label{lesl-et-t}
If $l$ is even $($resp. odd$)$, then $l = 2 \ ($resp. $l=2^s-1$, with $s \geq 1)$.
\end{lemma}
\begin{proof}
$\bullet$ If $l$ is even and $l \geq 4$, then put $l = 2n, \ n \geq 2$. As above, $\sigma(R^n)$ and  $\sigma(R^{n-1})$ divide $A$.\\
- If $n$ is even, then we must have $\sigma(R^n) =S^z$ because $P,Q$ divide $1+R$, $R$ does not divide $\sigma(R^n)$  and $\gcd(1+R,\sigma(R^n)) = 1$. Hence $z = 1$ and $S = \sigma(R^n)$ is irreducible. It is impossible.\\
- If $n$ is odd, then  $\sigma(R^{n-1}) =S$ which is impossible, as above.\\
$\bullet$ If $l = 2^ru-1$ is odd, with $u$ odd, then $\sigma^{**}(R^l) = \sigma(R^l) = (1+R)^{2^r-1} [\sigma(R^{u-1})]^{2^r}.$\\
If $u \geq 3$, then  $\sigma(R^{u-1}) = S$, which is impossible.
\end{proof}

\subsubsection{Case $\deg(P) >1$} \label{omega=4-no-splits}
Several proofs are similar to those in Section \ref{omega=4-one-splits}. 
As above, Lemma \ref{Q=P+1} implies that $Q = P+1$. 
We write: $A = P^h(P+1)^kR^lS^t$. 
\begin{lemma} \label{caspair=6-2}
If $1+P+\cdots+ P^{2w} = RS$, then $\deg(R) = \deg(S)$, $2w=6$, $P \in \Omega_2$ and $R, S \in \{P^3+P+1, P^3+P^2+1\}$.
\end{lemma}
\begin{proof}
Suppose that $1+P+\cdots+ P^{2w} = RS$. One has $1+x+\cdots+x^{2w} = UV$ where $U(P) = R$ and $V(P) = S$. By Lemma \ref{caspair=6},
one has:
$U(0)=1=V(0),\ \deg(U) = \deg(V)$. 
So,  $\deg(R) = \deg(S)$.\\
Moreover, $U$ and $V$ must be of the form $x^u(x+1)^v + 1$. Indeed, if $1+U = x^{u_1}(x+1)^{v_1}L^z$, with $z \geq 0$, then $1+R = P^u(P+1)^vL(P)^z$, 
$L(P) = S^y$, $y \geq 1$, $\deg(S)=\deg(R) = u \deg(P) + zy \deg(S)$, $zy = 0$. Thus, $z = 0$ and  $1+U = x^{u_1}(x+1)^{v_1}$. 
Analogously, $1+V =  x^{u_2}(x+1)^{v_2}$. 
Therefore, by Lemma \ref{caspair=6}, $2w=6$ and $R,S \in \{P^3+P+1, P^3+P^2+1\}$.
\end{proof}
\begin{lemma} \label{caseheven2}
If $h$ is even, then $h \in \{2,14\}.$
\end{lemma}
\begin{proof}
$\bullet$ If  $h \in \{4,6\}$, then $P+\alpha$ and $P+\alpha+1$ both divide $\sigma^{**}(P^h)$ and thus, they divide $\sigma^{**}(A) = A$. 
So, $P$, $P+1$, $R=P+\alpha$ and $S=P+\alpha+1$  are all irreducible over $\F_4$, which is impossible.\\
$\bullet$ If $h = 2n \geq 8$, then  $\sigma^{**}(P^h) = (1+P) \sigma(P^n) \sigma(P^{n-1})$.\\
- If $n=2w \geq 4$ is even, then $\sigma(P^n) = RS$,  $\deg(R) = \deg(S)$. We obtain $2w = 6$ and $h=12$.\\
But $P+1, P+\alpha$ and $P+\alpha+1$ all divide $\sigma^{**}(P^{12})$. As above, it is impossible.\\
- If $n=2w+1$ is odd, then $\sigma(P^{n-1}) = RS$ and $n-1=2w = 6$. So, $h=14$ and $R,S \in \{P^3+P+1, P^3+P^2+1\}$.
\end{proof}
\begin{lemma}  \label{casehodd2}
If $h$ is odd, then $h = 2^ru-1$ where $r \in \N^*$ and $u \in \{1,7\}.$
\end{lemma}
\begin{proof}
Put $h = 2^ru-1$ with $u$ odd. One has: $$\sigma^{**}(P^h) = \sigma(P^h) = (1+P)^{2^r-1} [\sigma(P^{u-1})]^{2^r}.$$
If $u \geq 3$, then  $\sigma(P^{u-1}) = RS$ and as we have just seen above, $u-1=6$ and $R, S \in \{P^3+P+1, P^3+P^2+1\}$.
\end{proof}
We also  get the analoguos of Lemma \ref{lesl-et-t}.
\begin{lemma} \label{lesl-et-t-2}
If $l$ is even $($resp. odd$)$, then $l = 2 \ ($resp. $l=2^s-1$, with $s \geq 1)$.
\end{lemma}
\subsubsection{End of the proof}
We recapitulate below, for $P \in \Omega_2$, 
$Q = P+1$,  $R = P^3+P+1$ and $S = P^3+P^2+1$,  the expressions of $\sigma^{**}(T^z)$, for $T^z \in \{P^h,Q^k,R^l,S^t\}$.

Keep in mind that $h,k,l$ and $t$ are not all odd. 
\begin{equation} \label{expressionsigstar21}
\begin{array}{|l|c|}
\hline
h&\sigma^{**}(P^h)\\
\hline
2&Q^2\\
\hline
14&Q^8RS\\
\hline
2^r - 1&Q^{2^r-1}\\
\hline
7\cdot 2^r - 1&Q^{2^r-1}R^{2^r}S^{2^r}\\
\hline
\end{array} \ \begin{array}{|l|c|}
\hline
k&\sigma^{**}(Q^k)\\
\hline
2&P^2\\
\hline
14&P^8RS\\
\hline
2^s - 1&P^{2^s-1}\\
\hline
7\cdot 2^s - 1&P^{2^s-1}R^{2^s}S^{2^s}\\
\hline
\end{array}
\end{equation}~\\
\begin{equation} \label{expressionsigstar22}
\begin{array}{|l|c|}
\hline
l&\sigma^{**}(R^l)\\
\hline
2& P^2Q^4\\
\hline
2^e - 1& P^{2^e-1} \cdot Q^{2\cdot(2^e-1)}\\
\hline
\end{array} \ \begin{array}{|l|c|}
\hline
t&\sigma^{**}(S^t)\\
\hline
2&P^4Q^2\\
\hline
2^f - 1&P^{2\cdot(2^f-1)} \cdot Q^{2^f-1}\\
\hline
\end{array}
\end{equation}~\\
We compare from Tables (\ref{expressionsigstar21}) and (\ref{expressionsigstar22}), the exponents of $P,Q,R,S$ in $\sigma^{**}(A)$ and in $A$. 
Instead of considering several possible cases, we give an upper bound to each exponent $a \in \{h,k,l,t\}$. We use {\tt{Maple}} computations to determine those which satisfy $\sigma^{**}(A) = A$. We obtain the following results.

\begin{lemma} \label{lesexposants2}
- If $h$ and $k$ are both even, then $h,k \in \{2,14\}$ and $e, f \leq 3$. So, 
$l,t \in \{1,2,3,7\}$.\\
- If $h$ is even and $k$ odd, then $h \in \{2,14\}$ and $s,e,f \leq 3$. So,\\
$k \in \{1,3,7,13,27,55\}$ and $l,t \in \{1,2,3,7\}$.\\
- If $h$ and $k$ are both odd, then $(h,k,l,t) \in \{(3,7,2,2),(7,3,2,2)\}$. 
\end{lemma}
\begin{proof}
- If $h$ is even, then $h \leq 14$. Each exponent of $P$ in the tables equals at most $14$. So, $s,e,f \leq 3$.\\
- If $h=2^r-1$ and $k=2^s-1$, then $\sigma^{**}(P^hQ^k) = P^kQ^h$, $h=k$ and $P^hQ^k$ is b.u.p. Hence, $R^lS^t$ is also b.u.p. and $l=t=2$.\\
- If $h=2^r-1$ and $k=7 \cdot 2^s-1$, then only $R^{2^s}$ and $S^{2^s}$  divide  $\sigma^{**}(A) = A$. So, $s=1$, $l=t=2$, $k = 13$. Thus, 
$\sigma^{**}(R^lS^t) = P^6Q^6$. By comparing the exponents of $Q$ in the tables, we get $6+2^r-1 = k=13$. So, $r= 3$ and $h=7$.\\
Analogously, if  $h=7 \cdot 2^r-1$ and $k=2^s-1$, then $h=13, k=7, l=t=2$.\\
- If $h=7 \cdot 2^r-1$ and $k=2^s-1$, then  only $R^{2^r+2^s}$ and $S^{2^r+2^s}$  divide  $\sigma^{**}(A) = A$. So, $l=t=2$ and we get the contradiction:
$2^r+2^s =2$ with $r,s \geq 1$. 
\end{proof} 
We also remark that the values of the exponents $h,k,l$ and $t$ do not depend on the choice of $P \in \Omega_2$. 
Therefore, for the computations with {\tt{Maple}}, we took two values of $P$:  $P = x$ and $P=x^{2} + x + \alpha$.
\begin{proposition} \label{lesexposants23}
If $A=P^hQ^kR^lS^t$ is b.u.p and indecomposable, where $h,k,l$ and $t$  are not all odd, then
$$\mbox{$P \in \Omega_2$, $Q = P+1$. $R, S \in \{P^3+P+1, P^3+P^2+1\}$,}$$ and $(h,k,l,t) \in \{(7,13,2,2), (13,7,2,2), (14,14,2,2)\}$.
\end{proposition}

\subsubsection{Maple Computations} \label{compute}
We search 
all $A=P^hQ^k R^lS^t$ such that $h,k,l,t$ are not all odd, $\omega(A) \geq 3$ and
$\sigma^{**}(A)= A$, by means of Lemmas \ref{casehodd}, \ref{lesl-et-t}  and \ref{lesexposants2}. We get the results stated in Proposition  \ref{lesexposants23}.\\
$\bullet$ {\bf{$\alpha \in \F_4$ is defined as follows:}}
\begin{verbatim}
> alias(alpha = RootOf(x^2 + x + 1)):
\end{verbatim}
$\bullet$ {\bf{The function $\sigma^{**}$ is defined as}} Sigm2star
\begin{verbatim}
> Sigm2star1:=proc(S,a) if a=0 then 1;else if a mod 2 = 0
then n:=a/2:sig1:=sum(S^l,l=0..n):sig2:=sum(S^l,l=0..n-1):
Factor((1+S)*sig1*sig2,alpha) mod 2:
else Factor(sum(S^l,l=0..a),alpha) mod 2:fi:fi:end:
> Sigm2star:=proc(S) P:=1:L:=Factors(S,alpha) mod 2:k:=nops(L[2]):
for j to k do S1:=L[2][j][1]:h1:=L[2][j][2]:
P:=P*Sigm2star1(S1,h1):od:Expand(P) mod 2:end:
\end{verbatim}

\def\thebibliography#1{\section*{\titrebibliographie}
\addcontentsline{toc}
{section}{\titrebibliographie}\list{[\arabic{enumi}]}{\settowidth
 \labelwidth{[
#1]}\leftmargin\labelwidth \advance\leftmargin\labelsep
\usecounter{enumi}}
\def\newblock{\hskip .11em plus .33em minus -.07em} \sloppy
\sfcode`\.=1000\relax}
\let\endthebibliography=\endlist

\def\biblio{\def\titrebibliographie{References}\thebibliography}
\let\endbiblio=\endthebibliography

\newbox\auteurbox
\newbox\titrebox
\newbox\titrelbox
\newbox\editeurbox
\newbox\anneebox
\newbox\anneelbox
\newbox\journalbox
\newbox\volumebox
\newbox\pagesbox
\newbox\diversbox
\newbox\collectionbox
\def\fabriquebox#1#2{\par\egroup
\setbox#1=\vbox\bgroup \leftskip=0pt \hsize=\maxdimen \noindent#2}
\def\bibref#1{\bibitem{#1}


\setbox0=\vbox\bgroup}
\def\auteur{\fabriquebox\auteurbox\styleauteur}
\def\titre{\fabriquebox\titrebox\styletitre}
\def\titrelivre{\fabriquebox\titrelbox\styletitrelivre}
\def\editeur{\fabriquebox\editeurbox\styleediteur}

\def\journal{\fabriquebox\journalbox\stylejournal}

\def\volume{\fabriquebox\volumebox\stylevolume}
\def\collection{\fabriquebox\collectionbox\stylecollection}
{\catcode`\- =\active\gdef\annee{\fabriquebox\anneebox\catcode`\-
=\active\def -{\hbox{\rm
\string-\string-}}\styleannee\ignorespaces}}
{\catcode`\-
=\active\gdef\anneelivre{\fabriquebox\anneelbox\catcode`\-=
\active\def-{\hbox{\rm \string-\string-}}\styleanneelivre}}
{\catcode`\-=\active\gdef\pages{\fabriquebox\pagesbox\catcode`\-
=\active\def -{\hbox{\rm\string-\string-}}\stylepages}}
{\catcode`\-
=\active\gdef\divers{\fabriquebox\diversbox\catcode`\-=\active
\def-{\hbox{\rm\string-\string-}}\rm}}
\def\ajoutref#1{\setbox0=\vbox{\unvbox#1\global\setbox1=
\lastbox}\unhbox1 \unskip\unskip\unpenalty}
\newif\ifpreviousitem
\global\previousitemfalse
\def\separateur{\ifpreviousitem {,\ }\fi}
\def\voidallboxes
{\setbox0=\box\auteurbox \setbox0=\box\titrebox
\setbox0=\box\titrelbox \setbox0=\box\editeurbox
\setbox0=\box\anneebox \setbox0=\box\anneelbox
\setbox0=\box\journalbox \setbox0=\box\volumebox
\setbox0=\box\pagesbox \setbox0=\box\diversbox
\setbox0=\box\collectionbox \setbox0=\null}
\def\fabriquelivre
{\ifdim\ht\auteurbox>0pt
\ajoutref\auteurbox\global\previousitemtrue\fi
\ifdim\ht\titrelbox>0pt
\separateur\ajoutref\titrelbox\global\previousitemtrue\fi
\ifdim\ht\collectionbox>0pt
\separateur\ajoutref\collectionbox\global\previousitemtrue\fi
\ifdim\ht\editeurbox>0pt
\separateur\ajoutref\editeurbox\global\previousitemtrue\fi
\ifdim\ht\anneelbox>0pt \separateur \ajoutref\anneelbox
\fi\global\previousitemfalse}
\def\fabriquearticle
{\ifdim\ht\auteurbox>0pt        \ajoutref\auteurbox
\global\previousitemtrue\fi \ifdim\ht\titrebox>0pt
\separateur\ajoutref\titrebox\global\previousitemtrue\fi
\ifdim\ht\titrelbox>0pt \separateur{\rm in}\
\ajoutref\titrelbox\global \previousitemtrue\fi
\ifdim\ht\journalbox>0pt \separateur
\ajoutref\journalbox\global\previousitemtrue\fi
\ifdim\ht\volumebox>0pt \ \ajoutref\volumebox\fi
\ifdim\ht\anneebox>0pt  \ {\rm(}\ajoutref\anneebox \rm)\fi
\ifdim\ht\pagesbox>0pt
\separateur\ajoutref\pagesbox\fi\global\previousitemfalse}
\def\fabriquedivers
{\ifdim\ht\auteurbox>0pt
\ajoutref\auteurbox\global\previousitemtrue\fi
\ifdim\ht\diversbox>0pt \separateur\ajoutref\diversbox\fi}
\def\endbibref
{\egroup \ifdim\ht\journalbox>0pt \fabriquearticle
\else\ifdim\ht\editeurbox>0pt \fabriquelivre
\else\ifdim\ht\diversbox>0pt \fabriquedivers \fi\fi\fi
.\voidallboxes}

\let\styleauteur=\sc
\let\styletitre=\it
\let\styletitrelivre=\sl
\let\stylejournal=\rm
\let\stylevolume=\bf
\let\styleannee=\rm
\let\stylepages=\rm
\let\stylecollection=\rm
\let\styleediteur=\rm
\let\styleanneelivre=\rm

\begin{biblio}{99}

\begin{bibref}{BeardU}
\auteur{J. T. B. Beard Jr} \titre{Unitary perfect polynomials over
$GF(q)$} \journal{Rend. Accad. Lincei} \volume{62} \pages 417-422
\annee 1977
\end{bibref}

\begin{bibref}{BeardBiunitary}
\auteur{J. T. B. Beard Jr}  \titre{Bi-Unitary Perfect polynomials over $GF(q)$}
\journal{Annali di Mat. Pura ed Appl.} \volume{149(1)} \pages 61-68 \annee 1987
\end{bibref}

\begin{bibref}{Canaday}
\auteur{E. F. Canaday} \titre{The sum of the divisors of a
polynomial} \journal{Duke Math. J.} \volume{8} \pages 721-737 \annee
1941
\end{bibref}

\begin{bibref}{Gall-Rahav-F4}
\auteur{L. H. Gallardo, O. Rahavandrainy} \titre{On
perfect polynomials over $\F_4$}
\journal{Port. Math. (N.S.)} \volume{62(1)} \pages 109-122 \annee
2005
\end{bibref}

\begin{bibref}{Gall-Rahav3}
\auteur{L. Gallardo, O. Rahavandrainy} \titre{Perfect polynomials
over $\F_4$ with less than five prime factors} \journal{Port. Math.
(N.S.) } \volume{64(1)} \pages 21-38 \annee 2007
\end{bibref}

\begin{bibref}{Gall-Rahav8}
\auteur{L. H. Gallardo, O. Rahavandrainy} \titre{All perfect
polynomials with up to four prime factors over $\F_4$}
\journal{Math. Commun.} \volume{14(1)} \pages 47-65 \annee 2009
\end{bibref}

\begin{bibref}{Gall-Rahav9}
\auteur{L. H. Gallardo, O. Rahavandrainy} \titre{On unitary
splitting perfect polynomials over $\F_{p^2}$} \journal{Math.
Commun.} \volume{15(1)} \pages 159-176 \annee 2010
\end{bibref}

\begin{bibref}{Gall-Rahav2}
\auteur{L. H. Gallardo, O. Rahavandrainy} \titre{On splitting
perfect polynomials over $\F_{p^p}$} \journal{Int. Electron. J.
Algebra} \volume{9} \pages 85-102 \annee 2011
\end{bibref}

\begin{bibref}{Gall-Rahav10}
\auteur{L. H. Gallardo, O. Rahavandrainy} \titre{Unitary perfect
polynomials over $\F_4$ with less than five prime factors}
\journal{Funct. et Approx.} \volume{45(1)} \pages 67-78 \annee 2011
\end{bibref}

\begin{bibref}{Gall-Rahav-bup-Mersenne}
\auteur{L. H. Gallardo, O. Rahavandrainy} \titre{All bi-unitary perfect polynomials over $\F_2$ only divisible by $x$, $x+1$ and by Mersenne primes} \journal{arXiv Math: 2204.13337} \annee 2022
\end{bibref}

\begin{bibref}{Gall-Rahav-split-Fp2}
\auteur{L. H. Gallardo, O. Rahavandrainy} \titre{On splitting bi-unitary perfect polynomials over $\F_{p^2}$} 
\journal{arXiv Math: 2310.05540v3} \annee 2023
\end{bibref}

\end{biblio}

\end{document}